\documentclass[12pt,a4paper]{article}

\usepackage{amssymb}
\usepackage{amsthm}
\usepackage{amsmath}
\usepackage[colorlinks=true,allcolors=black]{hyperref}
\usepackage{enumerate}

\newtheorem{theorem}{Theorem}[section]

\newtheorem{lemma}[theorem]{Lemma}
\newtheorem{proposition}[theorem]{Proposition}

\begin{document}

\title{Characterizing finite nilpotent groups associated with a graph theoretic equality}
\author{Ramesh Prasad Panda \and Kamal Lochan Patra \and Binod Kumar Sahoo}

\date{}

\maketitle

\begin{abstract}
The power graph of a group is the simple graph whose vertices are the group elements and two vertices are adjacent whenever one of them is a positive power of the other. We characterize the finite nilpotent groups whose power graphs have equal vertex connectivity and minimum degree.

\vskip .5cm

\noindent {\bf Key words.}   Nilpotent group, Power graph, Vertex connectivity, Minimal degree\\
\noindent {\bf AMS subject classification.} 20D15, 05C25, 05C40

\end{abstract}

\section{Introduction}

There are various graphs associated with groups which have been studied in the literature, e.g., Cayley graph, commuting graph, generating graph, prime graph. An interesting problem in this respect is to investigate the relation between a group and the associated graph. For instance, an algebraic characterization was given in \cite{Doty} for the generating sets of Cayley graphs with neighbor connectivity one for a class of groups including abelian groups. Whereas, in \cite{Gill}, the finite quasisimple groups with perfect commuting graphs were classified. 

The notion of the directed power graph of a group was introduced by Kelarev and Quinn in \cite{kel-2000} (also, see \cite{kel-2002} for a semigroup). The underlying undirected graph, simply termed as power graph, was first considered by Chakrabarty et al. in \cite{CGS-2009}. The {\it power graph} of a group $G$, denoted by $\mathcal{P}(G)$, is the simple graph with vertex set $G$ and two vertices $u$ and $v$ are adjacent if $u = v^k$ or $v = u^l$ for some positive integers $k$ and $l$.

By considering the aforementioned problem for power graphs, Cameron \cite{Cameron} proved that if the power graphs of two finite groups are isomorphic, then they have the same numbers of elements of each order. Curtin and Pourgholi \cite{curtin} showed that among all finite groups of a given order, the cyclic
group of that order has the maximum number of edges in its power graph.
Moreover, in \cite{MA2015}, Ma and Feng classified all finite groups whose power graphs are uniquely colorable, split or unicyclic.
In \cite{P-IM}, Panda characterized the finite noncyclic groups of prime exponent in terms of the edge connectivity of their power graphs.
In the present paper, working further on the relation between a group and its power graph, we characterize the finite nilpotent groups whose power graphs have equal vertex connectivity and minimum degree.

\subsection{Graphs and groups}

Let $\Gamma$ be a simple graph with vertex set $V$. A vertex $v\in V$ is called a {\it cut-vertex} of $\Gamma$ if the induced subgraph of $\Gamma$ on $V\setminus\{v\}$ is  disconnected. The {\it vertex connectivity} of $\Gamma$, denoted by $\kappa(\Gamma)$, is the minimum number of vertices whose deletion gives a disconnected or trivial subgraph of $\Gamma$. The latter case arises only when $\Gamma$ is a complete graph. Analogously, the {\it edge connectivity} of $\Gamma$, denoted by $\kappa'(\Gamma)$, is the minimum number of edges whose deletion gives a disconnected subgraph of $\Gamma$.
For $v\in V$, we denote by $\deg(v)$ the degree of $v$ in $\Gamma$. The {\it minimum degree} of $\Gamma$, denoted by $\delta(\Gamma)$, is defined by $\delta(\Gamma)=\min \{\deg(v):v\in V\}$.
It is known that $\kappa(\Gamma)\leq \kappa'(\Gamma)\leq \delta(\Gamma)$ and that $\kappa'(\Gamma)=\delta(\Gamma)$ if the diameter of $\Gamma$ is at most $2$.

Let $G$ be a group with identity element $e$. The order of an element $x\in G$ is denoted by $o(x)$. Recall that $G$ is said to be {\it nilpotent} if its lower central series $G=G_1 \geq G_2 \geq G_3 \geq G_4 \geq \cdots $ terminates at $\{e\}$ after a finite number of steps, where $G_{i+1}=[G_i, G]$ for $i\geq 1$. If $G$ is a finite group, then $G$ being nilpotent is equivalent to any of the following statements:
\begin{enumerate}[\rm(a)]
\item Every Sylow subgroup of $G$ is normal (equivalently, there is a unique Sylow $p$-subgroup of $G$ for every prime divisor $p$ of $|G|$).
\item $G$ is the direct product of its Sylow subgroups.
\item For $x,y\in G$, $x$ and $y$ commute whenever $o(x)$ and $o(y)$ are relatively prime.
\end{enumerate}
Every abelian group is nilpotent. Also, every finite group of prime power order is nilpotent.
A cyclic subgroup $M$ of $G$ is called a {\it maximal cyclic subgroup} if it is not properly contained in any cyclic subgroup of $G$. We denote by $\langle x\rangle$ the cyclic subgroup of $G$ generated by an element $x\in G$ and by $\mathcal{M}(G)$ the collection of all maximal cyclic subgroups of $G$. If $G$ is finite, then every element of $G$ is contained in a maximal cyclic subgroup. This statement, however, need not hold in an infinite group; $(\mathbb{Q},+)$ is an example.
For every positive integer $n\geq 2$, the {\it dicyclic group} of order $4n$ is given by $Q_{4n}:=\langle x,y\mid x^{2n}=e, x^n=y^2, y^{-1}xy=x^{-1}\rangle$. This group has a unique element of order $2$, namely $y^2=x^n$. If $n$ is a power of $2$, then $Q_{4n}$ is known as a {\it generalized quaternion group}.

\subsection{Known results}

Let $G$ be a finite group. Observe that two distinct vertices $x,y\in G$ are adjacent in $\mathcal{P}(G)$ if and only if $x\in\langle y\rangle$ or $y\in\langle x\rangle$. We note that if $\langle x \rangle = \langle y \rangle$, then  $\deg(x) = \deg(y)$. Further, if $\langle x\rangle\in \mathcal{M}(G)$, then $\deg(x)=o(x)-1$. For a nonempty proper subset $A$ of $G$, $\mathcal{P}(A)$ denotes the induced subgraph of $\mathcal{P}(G)$ with vertex set $A$. In that case, for $x\in A$, the degree of $x$ in $\mathcal{P}(A)$ is denoted by $\deg_{\mathcal{P}(A)}(x)$. If $A=G\setminus\{e\}$, then $\mathcal{P}(A)$ is called the {\it proper power graph} of $G$, denoted by $\mathcal{P}^{\ast}(G)$. The following result was proved in \cite[Corollary 4.1]{MRS-JAA}.

\begin{proposition}[\cite{MRS-JAA}]\label{proper-power}
Let $G$ be a finite group of prime power order. Then $\mathcal{P}^{\ast}(G)$ is connected if and only if $G$ is either cyclic or generalized quaternion.
\end{proposition}

For any finite group $G$, the identity element $e$ is adjacent to all other vertices of $\mathcal{P}(G)$. Thus $\mathcal{P}(G)$ is a connected graph with diameter at most 2 and hence $\kappa'(\mathcal{P}(G))=\delta(\mathcal{P}(G))$. The equality of $\kappa(\mathcal{P}(G))$ and $\delta(\mathcal{P}(G))$, however, need not hold in general: $\mathcal{P}(Q_{4n})$ is an example \cite[Theorem 6.9(ii)]{PK-CA}. Therefore, going forward, our objective will be to study this equality. By  \cite[Theorem 2.12]{CGS-2009}, $\mathcal{P}(G)$ is a complete graph if and only if $G$ is cyclic $p$-group for some prime $p$. In that case, $\kappa(\mathcal{P}(G)) =|G|-1= \delta(\mathcal{P}(G))$. The following proposition was proved in \cite[Theorem 6.2(ii)]{PK-CA} which gives a necessary condition for $\kappa(\mathcal{P}(G))=\delta(\mathcal{P}(G))$ when $G$ is not a cyclic $p$-group.

\begin{proposition}[\cite{PK-CA}]\label{EqualityNecessary}
Let $ G $ be a finite group different from a cyclic group of prime power order. Suppose that $\kappa(\mathcal{P}(G)) = \delta(\mathcal{P}(G))$. If $\delta(\mathcal{P}(G))=\deg(x)$ for some $x\in G$, then $o(x)=2$ and hence $G$ is of even order.
\end{proposition}

For a finite cyclic group $G$, the following proposition was proved in  \cite[Theorem 6.7]{PK-CA} giving both necessary and sufficient conditions for $\kappa(\mathcal{P}(G))=\delta(\mathcal{P}(G))$.

\begin{proposition}[\cite{PK-CA}]\label{EqCyclic}
Let $G$ be a finite cyclic group. Then $\kappa(\mathcal{P}(G)) = \delta(\mathcal{P}(G))$ if and only if $|G|$ is a prime power or twice of a prime power.
\end{proposition}

\subsection{Main result}

In this paper, we prove the following theorem which characterizes the equality of the vertex connectivity and the minimum degree in the power graphs of finite nilpotent groups.

\begin{theorem}\label{ThmEquality}
Let $G$ be a finite nilpotent group. Then $\kappa(\mathcal{P}(G)) = \delta(\mathcal{P}(G))$ if and only if one of the following holds:
\begin{enumerate}[\rm(1)]
	\item $G$ is cyclic group whose order is a prime power or twice of a prime power.
	\item $G$ is a noncyclic group of order $2^s p^{t}$, where $p$ is an odd prime, $s\geq 2$ and $t\geq 0$ are integers, satisfying the following:
	\begin{enumerate}[\rm(i)]
		\item The Sylow $2$-subgroup of $G$ is noncyclic and it contains a maximal cyclic subgroup of order $2$.
		\item The Sylow $p$-subgroup of $G$ is cyclic.
	\end{enumerate}
\end{enumerate}
\end{theorem}

We study the degrees of vertices of the power graphs of finite nilpotent groups in Section \ref{vertex-degrees} and then prove Theorem \ref{ThmEquality} in Section \ref{proof-main}.

\section{Vertex degrees}\label{vertex-degrees}

Throughout this section, $G$ is a finite nilpotent group of order $n$. We write the prime power factorization of $n$ as
$$n=p_1^{\alpha_1}p_2^{\alpha_2} \cdots p_r^{\alpha_r},$$
where $\alpha_1, \alpha_2, \ldots, \alpha_r$ are positive integers and $p_1,p_2,\dots, p_r$ are prime numbers with $p_1 < p_2 < \cdots < p_r$.
Denote by $P_i$ the unique Sylow $p_i$-subgroup of $G$ for $i\in [r]:=\{1,2,\dots,r\}$. Then $|P_i|=p_i^{\alpha_i}$ for $i\in [r]$ and $G=P_1P_2\cdots P_r$, the internal direct product of $P_1,P_2,\ldots, P_r$. For a given $x \in G$, there exists a unique element in $P_i$, denoted by $x_i$, for every $i\in [r]$ such that
$x = x_1x_2 \cdots x_r$.
Then $\langle x\rangle=\langle x_1\rangle \langle x_2\rangle\cdots \langle x_r\rangle$. Further, $\langle x\rangle\in\mathcal{M}(G)$ if and only if $\langle x_i\rangle\in\mathcal{M}(P_i)$ for every $i\in [r]$, see \cite[Lemma 2.11]{CPS-2}. Therefore, $\langle x\rangle\in\mathcal{M}(G)$ is of minimum order if and only if $\langle x_i\rangle\in\mathcal{M}(P_i)$ is of minimum order for every $i\in [r]$.

For $x\in G$, let $\tau_x$ denotes the subset of $[r]$ consisting of those elements $j\in [r]$ for which the $j$-th component $x_j$ of $x$ is nonidentity, that is,
$$\tau_x:=\{j\in [r]:\;x_j\neq e\}.$$
Note that $\tau_x=[r]$ if $\langle x\rangle\in\mathcal{M}(G)$ and that $\tau_x$ is empty if and only if $x=e$. We have $\deg(e)=|G|-1=n-1\geq\deg(x)$ for every $x\in G\setminus\{e\}$.
The following lemma helps us to compute the degrees of vertices of $\mathcal{P}(G)$.

\begin{lemma}\label{nbd-0}
For $x \in G\setminus\{e\}$, the sets
$A_x:=\left\{y \in G : y \neq x \text{ and } x \in \left\langle \underset{i\in \tau_x}\prod y_i \right\rangle\right \}$ and $B_x:=\{y \in G : y \neq x \text{ and } x \in \langle y \rangle \}$ are equal.
\end{lemma}

\begin{proof}
Clearly, $A_x =B_x$ if $\tau_x =[r]$. So consider that $\tau_x\subsetneq [r]$. Let $y \in A_x$. Then $x = \left(\underset{i\in \tau_x}\prod y_i\right)^m$ for some integer $m$. Define the integers $q: = \underset{j\in [r]\setminus \tau_x}\prod o(y_j)$ and $l: = mq$. Then $y^l = \left(\underset{i\in \tau_x}\prod y_i\right)^l = x^{q}$. Since $q$ and $o(x)$ are relatively prime, we have $\langle x \rangle=\langle x^q \rangle=\langle y^l \rangle$. This implies that $x\in\langle y\rangle$ and so $y\in B_x$. Thus $A_x\subseteq B_x$.

Conversely, consider an element $z \in B_x$. Then $x = z^t=z_1^{t}z_2^{t}\cdots z_r^{t}$ for some integer $t$. This gives
\begin{equation}\label{eqn-1}
	x \left(\underset{i\in \tau_x}\prod z_i^{-t} \right)= \underset{j\in [r]\setminus \tau_x}\prod z_j^{t}.
\end{equation}
Since the left and right hand sides of (\ref{eqn-1}) are elements of $\underset{i\in \tau_x}\prod P_i$ and $\underset{j\in [r]\setminus \tau_x}\prod P_j$ respectively, it follows that $x \left(\underset{i\in \tau_x}\prod z_i^{-t}\right)= e$. So $x=\underset{i\in \tau_x}\prod z_i^{t}=\left(\underset{i\in \tau_x}\prod z_i\right)^{t}$ and hence $x\in \left\langle \underset{i\in \tau_x}\prod z_i \right\rangle$. This gives $z\in A_x$. Thus $B_x\subseteq A_x$.
\end{proof}

We denote by $\phi$ the Euler's totient function. Then $\phi(p^k)=p^{k-1}(p-1)=p^{k-1}\phi(p)$ for any prime $p$ and positive integer $k$, and $\underset{d\mid m}{\sum} \phi(d) = m$ for every positive integer $m$. Recall that $\phi$ is a multiplicative function, that is, $\phi(ab)=\phi(a)\phi(b)$ for any two positive integers $a,b$ which are relatively prime. So $\phi(n)= \phi\left(p_1^{\alpha_1}\right)\phi\left(p_2^{\alpha_2}\right) \cdots \left(p_r^{\alpha_r}\right)$. 
The following proposition gives an expression for the degree of a nonidentity vertex of $\mathcal{P}(G)$.

\begin{proposition}\label{rem}
For $x \in G\setminus \{e\}$, we have
$$\deg(x) = o(x) - \phi(o(x))  + \left(\underset{j\in [r]\setminus \tau_x}\prod p_j^{\alpha_j} \right)\times \left |\left\{y \in \underset{i\in \tau_x}\prod P_{i}: x \in \langle y \rangle \right\}\right |-1.$$
\end{proposition}

\begin{proof}
Let $A_x$ and $B_x$ be the subsets of $G$ as defined in Lemma \ref{nbd-0}, and let $C_x$ be the set of all nongenerators of the cyclic subgroup $\langle x\rangle$ of $G$. Note that the neighbourhood of $x$ in $\mathcal{P}(G)$ is a disjoint union of $B_x$ and $C_x$. So $\deg(x)=|B_x|+|C_x|=|A_x|+|C_x|$, where the last equality holds by Lemma \ref{nbd-0}. We have $|C_x| = o(x) - \phi(o(x))$, and
\begin{align*}
	|A_x| &= \left|\left\{y \in G : y \neq x \text{ and } x \in \left\langle \underset{i\in \tau_x}\prod y_i \right\rangle\right \}\right|\\
	& = \left|\left\{y \in G : x \in \left\langle \underset{i\in \tau_x}\prod y_i \right\rangle\right \}\right| - 1 \nonumber \\
	& = \left(\underset{j\in [r]\setminus \tau_x}\prod p_j^{\alpha_j} \right) \times \left |\left\{y \in \underset{i\in \tau_x}\prod P_{i}: x \in \langle y \rangle \right\}\right |-1.
\end{align*}
Consequently, the result follows.
\end{proof}

We now compute a lower bound for the degrees of nonidentity vertices of $\mathcal{P}(G)$ and also provide necessary and sufficient conditions in order to attain that bound.

\begin{proposition}\label{degprime}
Consider $x \in G\setminus\{e\}$ and $M\in\mathcal{M}(G)$ containing $x$. If $|M|=p_{1}^{\gamma_{1}}p_{2}^{\gamma_{2}} \cdots p_{r}^{\gamma_{r}}$ and $o(x)=\underset{i\in \tau_x}\prod p_i^{\beta_i}$, where $1\leq \gamma_j\leq \alpha_j$ for $j\in [r]$ and $1\leq \beta_i\leq \gamma_i$ for $i\in \tau_x$, then
\begin{equation}\label{eqthm}
	\deg(x) \geq  o(x) - \phi(o(x))
	+ \left( \underset{j\in[r]\setminus \tau_x} \prod p_j^{\alpha_j} \right) \left(\underset{i\in \tau_x}\prod \left(p_{i}^{\gamma_{i}} - p_{i}^{\beta_i-1}\right)\right)  - 1.
\end{equation}
Further, equality holds in (\ref{eqthm}) if and only if $x$ belongs to exactly one maximal cyclic subgroup of $\underset{i\in \tau_x}\prod P_{i}$, namely $\underset{i\in \tau_x}\prod M_{i}$, where $M_i:=M\cap P_i\in \mathcal{M}(P_i)$ for $i\in [r]$.
\end{proposition}

\begin{proof}
We have $M = M_1 M_2 \cdots M_r$ and $x \in \underset{i\in \tau_x}\prod M_{i}\subseteq \underset{i\in \tau_x}\prod P_{i}$. In order to prove (\ref{eqthm}), it is enough to show using Proposition \ref{rem} that
\begin{equation}\label{eqthm-1}
	\left |\left\{y \in \underset{i\in \tau_x}\prod P_{i}: x \in \langle y \rangle \right\}\right |\geq \underset{i\in \tau_x}\prod \left(p_{i}^{\gamma_{i}} - p_{i}^{\beta_i-1}\right).
\end{equation}
Let $\tau_x=\{n_1,n_2,\ldots, n_s\}$ for some $s\in [r]$. Then
\begin{align}\label{eqthm-2}
	& \left |\left\{y \in \underset{i\in \tau_x} \prod M_{i}  : x \in \langle y \rangle \right\}\right |   = \sum_{\lambda_1 = \beta_{n_1}}^{\gamma_{n_1}} \sum_{\lambda_2 = \beta_{n_2}}^{\gamma_{n_2}} \cdots \sum_{\lambda_s = \beta_{n_s}}^{\gamma_{n_s}} \phi\left(p_{n_1}^{\lambda_{1}}p_{n_2}^{\lambda_{2}} \cdots p_{n_s}^{\lambda_{s}}\right) \nonumber\\
	& = \left(\sum_{\lambda_1 = \beta_{n_1}}^{\gamma_{n_1}} \phi\left(p_{n_1}^{\lambda_{1}}\right)\right) \left( \sum_{\lambda_2 = \beta_{n_2}}^{\gamma_{n_2}} \phi\left(p_{n_2}^{\lambda_{2}}\right)\right) \cdots \left(\sum_{\lambda_s = \beta_{n_s}}^{\gamma_{n_s}} \phi\left(p_{n_s}^{\lambda_{s}}\right)\right) \nonumber\\
	& = \left(p_{n_1}^{\gamma_{n_1}} - p_{n_1}^{\beta_{n_1}-1}\right)\left(p_{n_2}^{\gamma_{n_2}} - p_{n_2}^{\beta_{n_2}-1}\right)\cdots \left(p_{n_s}^{\gamma_{n_s}} - p_{n_s}^{\beta_{n_s}-1}\right)\nonumber\\
	& = \underset{i\in \tau_x}\prod \left(p_{i}^{\gamma_{i}} - p_{i}^{\beta_i-1}\right).
\end{align}
Now (\ref{eqthm-1}) follows from (\ref{eqthm-2}) and the fact that $\left\{y \in \underset{i\in \tau_x}\prod M_{i}: x \in \langle y \rangle \right\}\subseteq \left\{y \in \underset{i\in \tau_x}\prod P_{i}: x \in \langle y \rangle \right\}$. Next, we can see that equality holds in (\ref{eqthm}) if and only if the sets $\left\{y \in \underset{i\in \tau_x}\prod M_{i}: x \in \langle y \rangle \right\}$ and $\left\{y \in \underset{i\in \tau_x}\prod P_{i}: x \in \langle y \rangle \right\}$ are equal. The latter holds if and only if $\underset{i\in \tau_x}\prod M_{i}$ is the only maximal cyclic subgroup of $\underset{i\in \tau_x}\prod P_{i}$ containing $x$.
\end{proof}

The following proposition determines certain vertices of minimum degree among the vertices contained in a Sylow subgroup of $G$. 

\begin{proposition}\label{mindeg_in_pi}
For $k\in [r]$, suppose that $\langle y_k \rangle \in \mathcal{M}(P_k)$ is of minimum order. Then $\deg(y_k)\leq \deg(x_k)$ in $\mathcal{P}(G)$ for every $x_k\in P_k$.
\end{proposition}

\begin{proof}
Let $x_k \in P_k\setminus \{e\}$. Consider $ \langle z_k \rangle \in \mathcal{M}(P_k)$ containing $x_k$. Put $o(x_k)=p_k^\beta$ and $o(z_k)=p_k^\gamma$ for some positive integers $\beta\leq \gamma$. Applying Proposition \ref{degprime}, that is, inequality (\ref{eqthm}) for $\deg(x_k)$ and equality case for $\deg(z_k)$, we get
\begin{align}
	\deg(x_k) - \deg(z_k) & \geq p_k^{\beta-1} - p_k^{\gamma-1} +  \left( \underset{i\in[r]\setminus\{k\}}\prod p_i^{\alpha_i} \right)\left(p_k^{\gamma-1} - p_k^{\beta-1}\right) \nonumber\\
	& = \left( \underset{i\in[r]\setminus\{k\}}\prod p_i^{\alpha_i} -1 \right)\left(p_k^{\gamma-1} - p_k^{\beta-1}\right) \geq 0,\nonumber
\end{align}
giving $\deg(z_k)\leq \deg(x_k)$. Put $o(y_k)=p_k^{\lambda}$ for some positive integer $\lambda$. Since $o(y_k) \leq o(z_k)$, we have $\lambda\leq \gamma$. Applying the equality case of Proposition \ref{degprime} again, we get
$$\deg(z_k) - \deg(y_k) = \left(p_k^{\gamma-1} - p_k^{\lambda-1}\right) \left[(p_k -1) \left( \underset{i\in[r]\setminus\{k\}}\prod p_i^{\alpha_i}\right) +1 \right] \geq 0,$$
giving $\deg(y_k)\leq \deg(z_k)$. Hence $\deg(y_k)\leq \deg(x_k)$.
\end{proof}

Under certain conditions, the following proposition gives a strict upper bound for $\delta(\mathcal{P}(G))$ when $G$ is of even order.

\begin{proposition}\label{to-use}
Let $G$ be of even order. Suppose that the Sylow $2$-subgroup of $G$ is noncyclic containing a maximal cyclic subgroup of order $2$ and that all other Sylow subgroups of $G$ are cyclic. If $r\geq 3$, then
$\delta(\mathcal{P}(G))< \underset{i=2}{\overset{r}\prod} p_i^{\alpha_i}.$
\end{proposition}

\begin{proof}
Since $n$ is even, we have $p_1=2$ and $P_1$ is the Sylow $2$-subgroup of $G$. Let $k\in\{2,3,\ldots,r\}$ be such that $$p_k^{\alpha_k}=\min\{p_i^{\alpha_i}:2\leq i\leq r\}.$$
Since $r\geq 3$, the set $[r]\setminus \{1,k\}$ is nonempty. We have $p_j^{\alpha_j} > p_k^{\alpha_k}$ for every $j\in [r]\setminus \{1,k\}$ as $p_j$ and $p_k$ are distinct primes and so
\begin{equation}\label{eqn-2}
	\left( \underset{j\in [r]\setminus\{1,k\}}\prod p_j^{\alpha_j} \right) p_k^{\alpha_k -1} - \left(p_k^{\alpha_k} + p_k^{\alpha_k -1}\right)>0.
\end{equation}
The strict inequality in (\ref{eqn-2}) holds as the first term is an odd integer and the second term is an even integer.

For $i\in\{2,3,\ldots,r\}$, each $P_i$ is cyclic by the given hypothesis. Consider a generator $x_k$ of $P_k$ and an element $x_1\in P_1$ such that $\langle x_1\rangle$ is a maximal cyclic subgroup of $P_1$ of order $2$. Then $o(x_1x_k)=2p_k^{\alpha_k}$ and $\langle x_1x_k\rangle$ is a maximal cyclic subgroup of $P_1P_k$. Any maximal cyclic subgroup of $G$ containing $x_1x_k$ is of order $2p_{2}^{\alpha_{2}} \cdots p_{k}^{\alpha_{k}}\cdots p_{r}^{\alpha_{r}}$. By the equality case of Proposition \ref{degprime}, we get
\begin{align*}
	\deg(x_1x_k) 
	& = p_k^{\alpha_k} + p_k^{\alpha_k -1} +  \left( \underset{j\in [r]\setminus\{1,k\}}\prod p_j^{\alpha_j} \right) \left(p_{k}^{\alpha_{k}} - p_{k}^{\alpha_k-1}\right) - 1.
\end{align*}
Then
\begin{align*}
	\underset{i=2}{\overset{r}\prod} p_i^{\alpha_i} - \deg(x_1x_k) & = \left( \underset{j\in [r]\setminus\{1,k\}}\prod p_j^{\alpha_j} \right) p_k^{\alpha_k -1} - \left(p_k^{\alpha_k} + p_k^{\alpha_k -1}\right) + 1 > 0,
\end{align*}
using (\ref{eqn-2}). Consequently, $\underset{i=2}{\overset{r}\prod} p_i^{\alpha_i} > \deg(x_1x_k)\geq \delta(\mathcal{P}(G))$.
\end{proof}

For $r=2$, under certain conditions, we next determine some vertices of $\mathcal{P}(G)$ attaining the minimum degree.

\begin{proposition}\label{propcomp}
Suppose that $r=2$, $P_1$ is noncyclic and $P_2$ is cyclic. Then the following hold for any $\langle y \rangle \in \mathcal{M}(G)$:
\begin{enumerate}[\rm(i)]
	\item $\deg(y) < \deg(y_2)$.
	\item If $x \in \langle y \rangle$ with $\tau_x=\{1,2\}$, then $\deg(x) > \deg(y_1)$. In particular, $\deg(y) >\deg(y_1)$.
\end{enumerate}
\end{proposition}

\begin{proof}
Since $P_2$ is cyclic, we have $o(y_2)= p_2^{\alpha_2}$ and $o(y)= p_1^{\gamma_1}p_2^{\alpha_2}$ for some integer $\gamma_1$ with $1\leq \gamma_1\leq \alpha_1-1$. Then $o(y_1)= p_1^{\gamma_1}$ and $\deg(y)=o(y)-1=p_1^{\gamma_1}p_2^{\alpha_2}-1$.\\

\noindent (i) By the equality case of Proposition \ref{degprime}, we have
$\deg(y_2) =p_2^{\alpha_2 -1} + p_1^{\alpha_1}p_2^{\alpha_2 -1}\left(p_{2}- 1\right)  - 1$. Then
$$\deg(y_2)-\deg(y) =p_1^{\gamma_1}p_2^{\alpha_2 -1}\left[p_1^{\alpha_1-\gamma_1}(p_2 -1) - p_2 \right] + p_2^{\alpha_2 -1}>0,$$
as $\alpha_1 > \gamma_1$. Hence $\deg(y_2) > \deg(y)$.\\

\noindent (ii) Put $o(x) = p_1^{\beta_1}p_2^{\beta_2}$, where $1\leq \beta_1\leq \gamma_1$ and $1\leq \beta_2\leq \alpha_2$. By Proposition \ref{degprime}, we have	
\begin{align*}
	\deg(y_1) = p_1^{\gamma_1-1} + p_1^{\gamma_1}p_2^{\alpha_2}-p_1^{\gamma_1 -1}p_2^{\alpha_2}-1
\end{align*}
and
\begin{align*}
	\deg(x) & \geq  p_1^{\beta_1}p_2^{\beta_2} -  \phi\left(p_1^{\beta_1}p_2^{\beta_2}\right)+ \left(p_1^{\gamma_1}-p_1^{\beta_1-1}\right)\left(p_2^{\alpha_2}-p_2^{\beta_2-1}\right)-1\\
	& = p_1^{\beta_1-1}p_2^{\beta_2-1}\left[p_1+p_2 - p_1^{\gamma_1-\beta_1 +1}- p_2^{\alpha_2-\beta_2 +1} \right] + p_1^{\gamma_1}p_2^{\alpha_2} -1.
\end{align*}
Then
\begin{align*}
	& \deg(x) - \deg(y_1)\\
	& \geq p_1^{\beta_1-1}p_2^{\beta_2-1} \left[p_1+p_2 - p_1^{\gamma_1-\beta_1+1} -p_2^{\alpha_2-\beta_2+1} +  p_1^{\gamma_1-\beta_1}p_2^{\alpha_2-\beta_2+1} \right] - p_1^{\gamma_1-1}\\
	& = p_1^{\beta_1-1}p_2^{\beta_2-1} \left[p_2 +  \left(p_1^{\gamma_1-\beta_1}-1\right)\left(p_2^{\alpha_2-\beta_2+1}-p_1\right)\right] - p_1^{\gamma_1-1}\\
	& \geq p_1^{\beta_1-1}p_2^{\beta_2-1} \left[p_2 +  p_1^{\gamma_1-\beta_1}-1\right] - p_1^{\gamma_1-1} > 0.
\end{align*}
Hence $\deg(x) > \deg(y_1)$.
\end{proof}

\begin{proposition}\label{main-5}
Suppose that $r=2$, $P_1$ is noncyclic and $P_2$ is cyclic. If $\langle y\rangle\in\mathcal{M}(G)$ is of minimum order, then $\delta(\mathcal{P}(G)) = \deg(y_1)$.
\end{proposition}

\begin{proof}
We have $|\langle y\rangle|=o(y) = p_1^{\gamma_1}p_2^{\alpha_2}$ for some integer $\gamma_1$ with $1\leq \gamma_1\leq \alpha_1-1$. Let $x\in G\setminus\{e\}$ be arbitrary. We show that $\deg(x)\geq \deg(y_1)$.

First assume that $|\tau_x|=2$. Consider $\langle z\rangle \in\mathcal{M}(G)$ containing $x$. Then $o(z_1)\geq o(y_1)$ and Proposition \ref{propcomp}(ii) implies that $\deg(x) > \deg(z_1)$. The equality case of Proposition \ref{degprime} gives that $\deg(y_1) = \dfrac{o(y_1)}{p_1} + \phi(o(y_1))p_2^{\alpha_2} - 1$ and $\deg(z_1) = \dfrac{o(z_1)}{p_1} + \phi(o(z_1))p_2^{\alpha_2} - 1$. Since $o(z_1)\geq o(y_1)$, it follows that $\deg(z_1) \geq \deg(y_1)$ and hence $\deg(x) > \deg(y_1)$. In particular, $\deg(y) >\deg(y_1)$.

Now assume that $|\tau_x|=1$. Then $x\in P_1$ or $x\in P_2$. So $\deg(x) \geq \min \{\deg(y_1), \deg(y_2)\}$ using Proposition \ref{mindeg_in_pi}. Since $\deg(y_2)>\deg (y)>\deg(y_1)$ by Proposition \ref{propcomp}, we get $\deg(x)\geq \min \{\deg(y_1), \deg(y_2)\}=\deg(y_1)$.
\end{proof}

\section{Proof of Theorem \ref{ThmEquality}}\label{proof-main}

The following proposition proves Theorem \ref{ThmEquality} for the power graphs of groups of prime power order.

\begin{proposition}\label{PropEqConMinDeg}
Let $G$ be a finite $p$-group, where $p$ is a prime. Then $\kappa(\mathcal{P}(G)) = \delta(\mathcal{P}(G))$ if and only if one of the following holds:
\begin{enumerate}[\rm(i)]
	\item $G$ is cyclic.
	\item $G$ is a noncyclic $2$-group containing a maximal cyclic subgroup of order $2$.
\end{enumerate}
\end{proposition}

\begin{proof}
If $G$ is a cyclic $p$-group, then $\mathcal{P}(G)$ is complete and so $\kappa(\mathcal{P}(G)) =|G|-1= \delta(\mathcal{P}(G))$. If $G$ is a noncyclic $2$-group containing a maximal cyclic subgroup $\langle x\rangle$ of order $2$ for some $x\in G$, then the facts that $\deg(x)=1$ and the identity element $e$ is a cut-vertex of $\mathcal{P}(G)$ imply $\kappa(\mathcal{P}(G)) =1= \delta(\mathcal{P}(G))$.

Conversely, assume that $\kappa(\mathcal{P}(G))=\delta(\mathcal{P}(G))$. We may consider that $G$ is noncyclic. Then Proposition \ref{EqualityNecessary} implies that $G$ is a $2$-group and $\delta(\mathcal{P}(G))=\deg(x)$ for some element $x\in G$ of order $2$. By \cite[Theorem 6.9]{PK-CA}, $G$ is not a generalized quaternion group. Then the proper power graph $\mathcal{P}^{\ast}(G)$ is disconnected by Proposition \ref{proper-power}. It follows that $e$ is a cut-vertex of $\mathcal{P}(G)$, implying $\kappa(\mathcal{P}(G))=1$. Thus $\deg(x)=\delta(\mathcal{P}(G))=\kappa(\mathcal{P}(G))=1$. Then it follows that $\langle x\rangle$ is a maximal cyclic subgroup $G$ of order $2$.
\end{proof}

We need the following result whose proof is same as that of \cite[Proposition 4.2]{CPS-2}. It was originally proved for noncyclic abelian groups.

\begin{proposition}[\cite{CPS-2}]\label{to-use-1}
Let $G$ be a noncyclic nilpotent group of order $2^s p^t$, where $p$ is an odd prime and $s,t$ are positive integers. If the Sylow $2$-subgroup of $G$ is neither cyclic nor generalized quaternion, and the Sylow $p$-subgroup of $G$ is cyclic, then $\kappa(\mathcal{P}(G))=p^t$.
\end{proposition}

\begin{proof}[{\bf Proof of Theorem \ref{ThmEquality}}]
Let $G$ be a finite nilpotent group. If $G$ is cyclic with order a prime power or twice of a prime power, then $\kappa(\mathcal{P}(G))= \delta(\mathcal{P}(G))$ by Proposition \ref{EqCyclic}.
Suppose that $G$ is a noncyclic group of order $2^s p^{t}$, where $p$ is an odd prime, $s\geq 2$ and $t\geq 0$ are integers, in which the Sylow $2$-subgroup of $G$ is noncyclic containing a maximal cyclic subgroup of order $2$ and the Sylow $p$-subgroup of $G$ is cyclic. If $t=0$, then $\kappa(\mathcal{P}(G))= \delta(\mathcal{P}(G))$ by Proposition \ref{PropEqConMinDeg}. Assume that $t\geq 1$. The Sylow $2$-subgroup of $G$ is not generalized quaternion as it contains a maximal cyclic subgroup of order $2$. Therefore, $\kappa(\mathcal{P}(G))=p^{t}$ by Proposition \ref{to-use-1}. Let $\langle w_1\rangle$ be a maximal cyclic subgroup of order $2$ of the Sylow $2$-subgroup of $G$ and $w_2$ be a generator of the Sylow $p$-subgroup of $G$. Then $\langle w_1w_2\rangle$ is a maximal cyclic subgroup of $G$ of minimum order. We have $\delta(\mathcal{P}(G)) = \deg(w_1)$ by Proposition \ref{main-5}. Since $o(w_1)=2$, the equality case of Proposition \ref{degprime} gives that $\deg(w_1)=p^{t}$. Thus $\delta(\mathcal{P}(G))=p^{t}=\kappa(\mathcal{P}(G))$.

Conversely, assume that $\kappa(\mathcal{P}(G))=\delta(\mathcal{P}(G))$. Put $|G|=n$. By Proposition \ref{PropEqConMinDeg}, we may consider that $n$ is not a prime power. Write $n=p_1^{\alpha_1}p_2^{\alpha_2} \cdots p_r^{\alpha_r},$
where $r\geq 2$, $p_1,p_2,\dots, p_r$ are prime numbers with $p_1 < p_2 < \cdots < p_r$ and $\alpha_1, \alpha_2, \ldots, \alpha_r$ are positive integers. Let $P_i$ denote the Sylow $p_i$-subgroup of $G$ for $i\in [r]$.

Let $\delta(\mathcal{P}(G))= \deg(x)$ for some $x\in G$. Since $n$ is not a prime power, we have $o(x)=2$ by Proposition \ref{EqualityNecessary}. So $p_1=2$ and $G$ is of even order. Let $\langle y \rangle$ be a maximal cyclic subgroup of the Sylow $2$-subgroup $P_1$ of $G$ containing $x$. We have
$\deg(x)=\delta(\mathcal{P}(G)) \leq \deg(y).$
Since $\langle y\rangle$ is a cyclic $2$-group, the induced subgraph $\mathcal{P}(\langle y \rangle)$ of $\mathcal{P}(G)$ is complete and so $\deg_{\mathcal{P}(\langle y\rangle)}(x)= \deg_{\mathcal{P}(\langle y\rangle)}(y)$. Further, any element of $G\setminus\langle y\rangle$ that is adjacent to $y$ must be adjacent with $x$. It then follows that
$\deg(y)\leq  \deg(x).$
We thus have $\deg(y)=\deg(x)=\delta(\mathcal{P}(G))$. Then Proposition \ref{EqualityNecessary} again implies that $o(y)=2$ and hence $x=y$.
Thus $\langle x\rangle$ is a maximal cyclic subgroup of $P_1$ of order $2$. Then by the equality case of Proposition \ref{degprime}, we get
\begin{equation}\label{eqn-3}
	\delta(\mathcal{P}(G))=\deg(x)= \underset{i=2}{\overset{r}\prod} p_i^{\alpha_i}
\end{equation}
using the fact that $o(x)=2$.

Let $\langle z \rangle$ be a maximal cyclic subgroup of $G$ containing $x$. Then $\deg(z)=o(z)-1$ and $\langle z \rangle\cap P_1$ is a maximal cyclic subgroup of $P_1$ containing $x$. Replacing $\langle y\rangle$ by $\langle z \rangle\cap P_1$ in the previous paragraph, we get $\langle z \rangle\cap P_1=\langle x\rangle$.
We have $o(z) = 2p_2^{\gamma_2} \cdots p_r^{\gamma_r}$, where $1\leq \gamma_i\leq \alpha_i$ for $i\in\{2,3,\ldots, r\}$. Then
\begin{equation*}\label{EqEquality2}
	2p_2^{\gamma_2} \cdots p_r^{\gamma_r}> 2p_2^{\gamma_2} \cdots p_r^{\gamma_r} -1 =\deg(z)\geq \delta(\mathcal{P}(G))=\deg(x) =\underset{i=2}{\overset{r}\prod} p_i^{\alpha_i}
\end{equation*}
gives that
$\underset{i=2}{\overset{r}\prod} p_i^{\alpha_i - \gamma_i} < 2.$
This is possible if and only if $\alpha_i = \gamma_i$ for every $i\in \{2,3,\ldots,r\}$. It then follows that $P_2, P_3, \ldots, P_r$ are all cyclic.
If $P_1$ is also cyclic, then $G$ itself is cyclic. Consequently, $r=2$ and $|G|=n=2p_2^{\alpha_2}$ by Proposition \ref{EqCyclic}.

Now assume that $P_1$ is noncyclic. Then $G$ is noncyclic. If $r\geq 3$, then $\delta(\mathcal{P}(G))< \underset{i=2}{\overset{r}\prod} p_i^{\alpha_i}$ by Proposition \ref{to-use}, which contradicts (\ref{eqn-3}). Thus $r=2$ and $|G|=2^{\alpha_1}p_2^{\alpha_2}$, where the Sylow $2$-subgroup $P_1$ of $G$ is noncyclic containing a maximal cyclic subgroup $\langle x\rangle$ of order $2$ and the Sylow $p_2$-subgroup $P_2$ of $G$ is cyclic. This completes the proof.
\end{proof}

\section*{Acknowledgment}
The first author would like to thank the National Institute of Science Education and Research, Bhubaneswar, for the facilities provided when working as a visiting fellow in the School of Mathematical Sciences. The last author is partially supported by Project No. MTR/2017/000372 of the Science and Engineering Research Board (SERB), Department of Science and Technology, Government of India.





\vskip .5cm

\noindent\underline{\bf Addresses}:

\noindent{\bf Ramesh Prasad Panda ({\tt rppanda@niser.ac.in})\\
Kamal Lochan Patra ({\tt klpatra@niser.ac.in})\\
Binod Kumar Sahoo ({\tt bksahoo@niser.ac.in})}
\begin{enumerate}
\item[1)] School of Mathematical Sciences, National Institute of Science Education and Research, Bhubaneswar, P.O.-Jatni, District-Khurda, Odisha-752050, India.

\item[2)] Homi Bhabha National Institute, Training School Complex, Anushakti Nagar, Mumbai-400094, India.
\end{enumerate}
\end{document}